\documentclass[12pt, a4paper]{amsart}

\usepackage{amsmath,amssymb}
\usepackage{accents}
\usepackage{graphicx}
\usepackage{verbatim}
\usepackage{enumitem}
\usepackage{tikz}
\usepackage{caption}
\usepackage{subcaption}
\usepackage{tabu}
\usepackage{longtable}
\usepackage{array}
\usepackage{hyperref}
\usepackage[normalem]{ulem}
\usepackage{algorithm}
\usepackage[noend]{algpseudocode}
\everymath{\displaystyle}

\usetikzlibrary{decorations.pathreplacing,angles,quotes}

\newlength{\dhatheight}

\theoremstyle{definition}

\newtheorem{theorem}{Theorem}[section]

\newtheorem{proposition}[theorem]{Proposition}
\newtheorem{lemma}[theorem]{Lemma}
\newtheorem{remark}[theorem]{Remark}
\newtheorem{corollary}[theorem]{Corollary}

\newtheorem{question}[theorem]{Question}
\newtheorem{claim}[theorem]{Claim}

\newcommand{\ch}[3][2]{ {{#2} \choose{#3}} }

\newlist{pcases}{enumerate}{1}
\setlist[pcases]{
  label={\em{Case~\arabic*:}}\protect\thiscase.~,
  ref=\arabic*,
  align=left,
  labelsep=0pt,
  leftmargin=0pt,
  labelwidth=0pt,
  parsep=0pt
}
\newcommand{\case}[1][]{%
  \if\relax\detokenize{#1}\relax
    \def\thiscase{}%
  \else
    \def\thiscase{~#1}%
  \fi
  \item
}

\everymath{\displaystyle}

\title{Random meander model for links}
\author{Nicholas Owad}
\author{Anastasiia Tsvietkova}

\begin{document}

\begin{abstract}

We suggest a new random model for links based on meander diagrams and graphs. We then prove that trivial links appear with vanishing probability in this model, no link $L$ is obtained with probability 1, and there is a lower bound for the number of non-isotopic knots obtained for a fixed number of crossings. A random meander diagram is obtained through matching pairs of parentheses, a well-studied problem in combinatorics. Hence tools from combinatorics can be used to investigate properties of random links in this model, and, moreover, of the respective 3-manifolds that are link complements in 3-sphere. We use this for exploring geometric properties  of a link complement. Specifically, we give expected twist number of a link diagram and use it to bound expected hyperbolic and simplicial volume of random links. The tools from combinatorics that we use include Catalan and Narayana numbers, and Zeilberger’s algorithm.

\


\

\end{abstract}

\maketitle

\section{Introduction}

In the recent years, there has been an increased interest in using probabilistic methods in low-dimensional geometry and topology. A number of models for 3-manifolds and links appeared, and were used to study their topological properties and various invariants. For example, Dunfield and W. Thurston \cite{DunfieldThurston} studied finite covers of random 3-manifolds, and Even-Zohar, Hass, Linial and Nowik \cite{EZHLN4} studied linking number and Casson invariants in random petaluma model for knots. One of the benefits of using random models is that they often allow one to check the typical behavior of an object (e.g. of a link) beyond well-studied families. In this paper, we introduce a new model for links called the {\em random meander link} model. For an overview of previously existing models for random knots and links, we direct the reader to the exposition by Even-Zohar \cite{EvenZohar}.

For our model, we use meander diagrams. Informally, a meander is a pair of curves in the plane, where one curve is assumed to be a straight line and the other one ``meanders'' back and forth over it. In a meander diagram, we assume that the ends of these two curves are connected.  The study of meanders dates back at least to Poincar\'e's work on dynamical systems on surfaces. Meanders naturally appear in various areas of mathematics (see, for example, combinatorial study by Franz and Earnshaw \cite{FE}), as well as in natural sciences. Every knot is known to have a meander diagram \cite[Theorem 1.2]{AST}, and we generalize these diagrams so that every link has one too. This is described  in Section \ref{sec:model}. It does not follow however that a random model that produces various meander diagrams will produce all knots or links: indeed, many distinct meander diagrams may represent isotopic (i.e. equivalent) links. We address this separately, as explained below.

The first important question about any random link model is whether it produces non-trivial links with high probability. The proof of this is often far from trivial. For example, for a \textit{grid walk model}, it was conjectured by Delbruck in 1961 that a random knot $K_n$ is knotted with high probability \cite{Delbruck}. And only in 1988 and 1989, two proofs of this appeared, one by Sumners and Whittington \cite{SW}, and another by Pippenger \cite{Pippenger}. It was also conjectured that a different model, called a \textit{polygonal walk model}, produces unknots with vanishing probability by Frischand and Wasserman in 1961 \cite{FW}. The first proof of this appeared in 1994, for Gaussian-steps polygons, due to Diao, Pippenger, and Sumners \cite{DPS}. For another, more recent random knot model, called \textit{petaluma model} for knots, the paper from 2018 by Even-Zohar, Hass, Linial and Nowik  \cite{EZHLN1} is mostly devoted to proving that the probability of obtaining every specific knot type decays to zero as the number of petals grows. Yet another recent work by Chapman from 2017 \cite{Chapman} studies \textit{random planar diagram model} and shows that such diagrams are non-trivial with high probability. This list of such results is not exhaustive. 

In Sections \ref{SecCircles} and \ref{Unlinks}, we prove that as the number of crossings in a link diagram grows, we obtain a non-trivial knot or link with probability 1 in the \textit{random meander model}. This is the first main result of this paper, given in Theorem \ref{thm:Rare}. While we use a programmed worksheet for Zielberger's algorithm as a shortcut (which is mathematically rigorous), the main part of the proof is theoretical rather than computer-assisted, and relies on observations concerning topology of knots, graphs and combinatorics. In particular, we show that a certain fragment ("a pierced circle") appears in our random meander link diagrams with high probability. This can be compared with Chapman's approach \cite{Chapman}: he shows that another fragment, similar to a trefoil, appears in random planar diagrams with high probability as well. As a corollary, we also show that as the number of crossings and components in a random diagram grows, no link is obtained with probability 1 in our random model, and therefore the model yields infinitely many non-isotopic links (Proposition \ref{nonisotopic}). We proceed to give a lower bound on the number of distinct (up to isotopy) knots produced by our model for a fixed number of crossings in Proposition \ref{prime}.

In Section \ref{VolumeEtc}, we give an application of this random model to low-dimensional topology by finding the expected number of twists in a random link diagram. This is Theorem \ref{twists}. This number is related to geometric properties of the respective link complement in 3-sphere, as is shown, for example, in the work on hyperbolic volume by Lackenby \cite{Lackenby} and Purcell \cite{Purcell}, or in the work on cusp volume by Lackenby and Purcell \cite{LackenbyPurcell2}. Applying the former paper by Lackenby and its extension concerning simplicial volume by Dasbach and Tsvietkova \cite{DT}, we give bounds for the expected hyperbolic or simplicial volume of links in Corollary \ref{volume}, in the spirit of Obeidin's observations \cite{Obeidin}. Together, Theorem \ref{twists} and Corollary \ref{volume} can be seen as the second main result of the paper. We conclude with some open questions concerning hyperbolicity and volume of random meander links.

One of the advantages of random meander link model is that constructing meander diagrams corresponds in a certain way to a well-known combinatorial problem about pairs of parentheses which are correctly matched. Thus the number of random meander diagrams has a simple formula in terms of Catalan numbers. Moreover, topological and diagrammatic properties of random meander links translate into combinatorial identities, and can be investigated using tools from combinatorics. For example, to prove that unlinks are rare, we use Poincar\'e's theorem about recurrence relations, and Zeilberg's algorithm that finds a polynomial recurrence for hypergeometric identities. Finding mathematical expectation for volume of a link complement involves summing Narayana numbers, another well-known series. This is perhaps an unexpected way to look at knot theory problems.

\section{Random meander link model}\label{sec:model}

We begin by providing the background required to define our random model.  The following result is attributed to Gauss.

\begin{theorem}\label{thm:AST}{\cite[Theorem 1.2]{AST}}
Every knot has a projection that can be decomposed into two sub-arcs such that each sub-arc never crosses itself.
\end{theorem}

By planar isotopy, one of these sub-arcs can be taken to be a subset of the $x$-axis, which we will call the {\em axis}. The resulting diagram is called {\em straight}.  The other sub-arc we call the {\em meander curve.}  We will call a segment of the meander curve  between two consecutive crossings an {\em arc}.  The {\em {complementary} axis} is the $x$-axis minus the axis.  Then every arc that makes up the meander curve either crosses the complementary axis or not.  If an arc does not cross the complementary axis, we call it a {\em contained arc} and if it does cross the complementary axis, it is an {\em uncontained arc}. A straight diagram is said to be a {\em meander diagram} if there are no uncontained arcs.  See Figure \ref{fig:AST}.

\begin{figure}[h]
\begin{center}
\begin{tikzpicture}[scale=.4]

\begin{scope}[xshift = -8cm, scale = .8]

\draw [ultra thick] (-.2,.-.2) to [out=40, in=-90] (1.5,2.4);

\draw [ultra thick] (-.2,-.2) to [out=180+40, in=90] (-1.2,-1.6) to [out=-90, in=180-30] (0,-3.2);

\draw [line width=0.13cm, white] (0,-3.2) to [out=-30, in=-90] (3.5,-.6) to [out=90, in=-20] (1.6 ,2) to [out=160, in=20] (-1.3 ,2.1);
\draw [ultra thick] (0,-3.2) to [out=-30, in=-90] (3.5,-.6) to [out=90, in=-20] (1.6 ,2) to [out=160, in=20] (-1.3 ,2.1);

\draw [ultra thick] (.2,-.2) to [out=-40, in=120] (1.05,-1);
\draw [ultra thick] (1.05,-1) to [out=300, in=70] (1.1,-2.3);

\draw [line width=0.13cm, white] (1.1,-2.3) to [out=-110, in=27] (0,-3.25) to [out=207, in=-90] (-3.5,-.6) to [out=90, in=200] (-1.6 ,2);
\draw [ultra thick] (1.1,-2.3) to [out=-110, in=27] (0,-3.25) to [out=207, in=-90] (-3.5,-.6) to [out=90, in=200] (-1.6 ,2);

\draw [line width=0.13cm, white] (-.2,.2) to [out=140, in=-90] (-1.5,2.4);
\draw [ultra thick] (-.2,.2) to [out=140, in=-90] (-1.5,2.4);

\draw [ultra thick] (-1.5,2.4) arc [radius=1.5, start angle=180, end angle= 0];
\end{scope}

\begin{scope}[xshift = -1cm, yshift=0cm, scale=1.1]

\draw [ultra thick] (3,0) arc [radius=1/2, start angle=180, end angle= 0];

\draw [ultra thick] (2,0) arc [radius=3/2, start angle=180, end angle= 0];

\draw [ultra thick] (-1,0) arc [radius=1, start angle=180, end angle= 0];
\draw [ultra thick] (-1,0) arc [radius=5/2, start angle=180, end angle= 360];

\draw [ultra thick] (0,0) arc [radius=3/2, start angle=180, end angle= 360];

\draw [ultra thick] (1,0) arc [radius=1/2, start angle=180, end angle= 360];

\draw [line width=0.13cm, white] (0.5,0) to (1.85,0);
\draw [line width=0.13cm, white] (2.15,0) to (3.5,0);

\draw [ultra thick] (0,0) to (1.8,0);
\draw [ultra thick] (2.2,0) to (3.8,0);
\draw [ultra thick] (4.2,0) to (5,0);

\draw [ultra thick, line cap = round] (5,0) to (5,0);
\draw [ultra thick, line cap = round] (0,0) to (0,0);

\end{scope}

\begin{scope}[ xshift=8cm, scale=1.2]

\draw [ultra thick] (0,0) arc [radius=5/2, start angle=180, end angle= 360];

\draw [ultra thick] (1,0) arc [radius=1/2, start angle=180, end angle= 0];
\draw [ultra thick] (1,0) arc [radius=3/2, start angle=180, end angle= 360];

\draw [ultra thick] (2,0) arc [radius=1/2, start angle=180, end angle= 360];

\draw [ultra thick] (3,0) arc [radius=3/2, start angle=180, end angle= 0];

\draw [ultra thick] (4,0) arc [radius=1/2, start angle=180, end angle= 0];

\draw [line width=0.13cm, white] (1.5,0) to (2.85,0);
\draw [line width=0.13cm, white] (4.15,0) to (5.5,0);

\draw [ultra thick] (0,0) to (.8,0);
\draw [ultra thick] (1.2,0) to (2.8,0);
\draw [ultra thick] (3.2,0) to (3.8,0);
\draw [ultra thick] (4.2,0) to (6,0);

\draw [ultra thick, line cap = round] (6,0) to (6,0);
\draw [ultra thick, line cap = round] (0,0) to (0,0);

\end{scope}

\end{tikzpicture}
\end{center}
\caption{Three diagrams of the figure-eight knot.  On the left, the standard diagram; in the center, a straight diagram; on the right, a meander diagram.  }\label{fig:AST}
\end{figure}
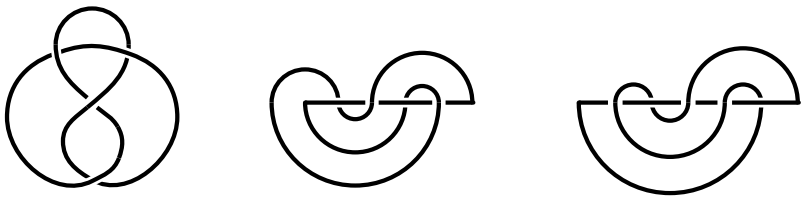

For study of meander knots and links from the point of view of knot theory and low-dimensional topology, see for example \cite{JR, Owad1, Owad2}. As we will see, variations of the constructions that we describe below (meander knot, meander link, meander link with multiple parallel strands) appeared in literature before, at times in a somewhat different form, but this is the first time they are used for a random model. The following fact was proved by Owad in \cite{Owad2}, Theorem 2.8, and was known to Adams, Shinjo, and Tanaka (see Fig. 2 in \cite{AST}). 

\begin{theorem}\label{thm:OwadMeander}
Every knot has a meander diagram.
\end{theorem}

A {\em meander graph} is the 4-valent planar graph obtained by replacing each crossing of a meander diagram with a vertex. The graph's edges are the upper semi-circles above the axis,
lower semi-circles below the axis, and every segment from a vertex to a vertex along the axis. In addition, there can be an edge of the graph made up of either the leftmost or the righmost segment of the axis and one semicircle, adjacent to that segment. Such an edge may be a loop, i.e. its two vertex endpoints may coincide. By a pair of parentheses we mean two parentheses, left and right: \texttt{()}. Further, we will refer to any string of $s$ pairs of parentheses as a \textit{p-string} of length $s$. Put a collection of upper semi-circles in correspondence with a  $p$-string, as in Figure  \ref{fig:parenLink}, where a pair $a$ is inside the pair $b$ if and only if the respective semicircle for $a$ is inside the respective semicircle for $b$. Do the same for lower semi-circles. To generate a random meander graph, we will use two $p$-strings, each of length $s$.

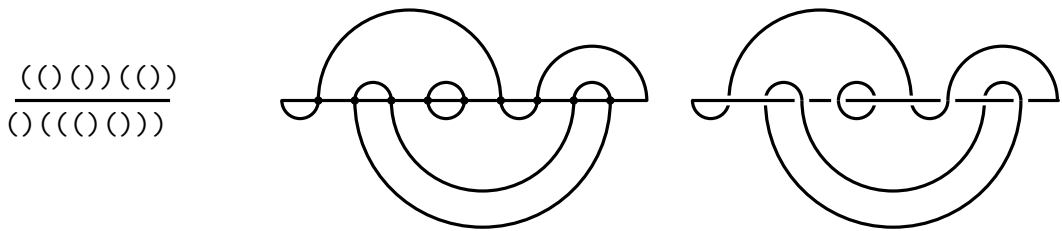
\begin{figure}[h]
\begin{center}
\begin{tikzpicture}[scale = .6]

\draw [very thick] (-5.85,0) to (-2.45,0);
\node [above] at (-4,0) {\texttt{(()())(())}};
\node [below] at (-4.29,0) { \texttt{()((()()))}};

\begin{scope}[scale = .8]
\draw [very thick, line cap=round] (0,0) to (10,0);
\draw [very thick] (10,0) arc [radius=1.5, start angle=0, end angle= 180];
\draw [very thick] (7,0) arc [radius=.5, start angle=0, end angle= -180];
\draw [very thick] (6,0) arc [radius=2.5, start angle=0, end angle= 180];
\draw [very thick] (1,0) arc [radius=.5, start angle=0, end angle= -180];

\draw [very thick] (9,0) arc [radius=.5, start angle=0, end angle= 180];
\draw [very thick] (8,0) arc [radius=2.5, start angle=0, end angle= -180];
\draw [very thick] (4,0) arc [radius=.5, start angle=180, end angle= 0];
\draw [very thick] (4,0) arc [radius=.5, start angle=180, end angle= 360];
\draw [very thick] (3,0) arc [radius=.5, start angle=0, end angle= 180];
\draw [very thick] (2,0) arc [radius=3.5, start angle=180, end angle= 360];

\foreach \a in {1,...,9}{
\draw [fill] (\a,0) circle [radius=0.1];
}
\end{scope}

\begin{scope}[xshift = 9cm, scale = .8]

\draw [very thick] (10,0) arc [radius=1.5, start angle=0, end angle= 180];
\draw [very thick] (7,0) arc [radius=.5, start angle=0, end angle= -180];
\draw [very thick] (6,0) arc [radius=2.5, start angle=0, end angle= 180];
\draw [very thick] (1,0) arc [radius=.5, start angle=0, end angle= -180];

\draw [very thick] (9,0) arc [radius=.5, start angle=0, end angle= 180];
\draw [very thick] (8,0) arc [radius=2.5, start angle=0, end angle= -180];
\draw [very thick] (4,0) arc [radius=.5, start angle=180, end angle= 0];
\draw [very thick] (4,0) arc [radius=.5, start angle=180, end angle= 360];
\draw [very thick] (3,0) arc [radius=.5, start angle=0, end angle= 180];
\draw [very thick] (2,0) arc [radius=3.5, start angle=180, end angle= 360];

\foreach \a in {.2}{
\draw [line width = .13cm, white] (\a,0) to (3-\a,0);
\draw [line width = .13cm, white] (3+\a,0) to (4-\a,0);
\draw [line width = .13cm, white] (4+\a,0) to (7-\a,0);
\draw [line width = .13cm, white] (7+\a,0) to (9-\a,0);
\draw [line width = .13cm, white] (9+\a,0) to (10-\a,0);

\draw [very thick] (0,0) to (3-\a,0);
\draw [very thick] (3+\a,0) to (4-\a,0);
\draw [very thick] (4+\a,0) to (7-\a,0);
\draw [very thick] (7+\a,0) to (9-\a,0);
\draw [very thick] (9+\a,0) to (10,0);

}
\draw [very thick, line cap=round] (0,0) to (0,0);
\draw [very thick, line cap=round] (10,0) to (10,0);

\end{scope}

\end{tikzpicture}
\caption{Left: two $p$-strings of length 5. Middle: the corresponding meander graph $\Gamma_9$.  Right: the random meander link $L_9$ obtained by adding crossing information string $V_9 = OOUUOOUOU$} to $\Gamma_9$.
\label{fig:parenLink}
\end{center}
\end{figure}

Let $\Gamma_{2s-1}$ be a meander graph generated with two $p$-strings of lengths $s$. The \textit{crossing information string} at each vertex is a word $V_{2s-1}$ consisting of letters $U$ and $O$ of length ${2s-1}$. The pair ($\Gamma_{2s-1}, V_{2s-1}$) defines a meander link $L_{2s-1}$ as follows: every letter in $V_{2s-1}$ from left to right corresponds to either overpass ($O$) or underpass ($U$) of the axis of the diagram drawn instead of the vertices of $\Gamma_{2s-1}$ from left to right.

Call the link component that contains the axis the {\em axis component}. All link crossings are a part of  this component. As a result, the other link components are unknots and not linked with each other.

Now we generalize this construction. For every link component of $L_{2s-1}$, take $r-1$ additional strands parallel to it. Instead of each crossing of $L_{2s-1}$ we have $r^2$ crossings in the new link. Substitute each of them by a vertex again. Denote the resulting graph by $\Gamma^r_{2s-1}$ and call it $(r,{2s-1})$-meander graph. Denote the collection of $r$ words, where each word consists of $r({2s-1})$ letters $O$ and $U$, by $V^r_{2s-1}$, and call it  $(r,{2s-1})$-crossing information string.

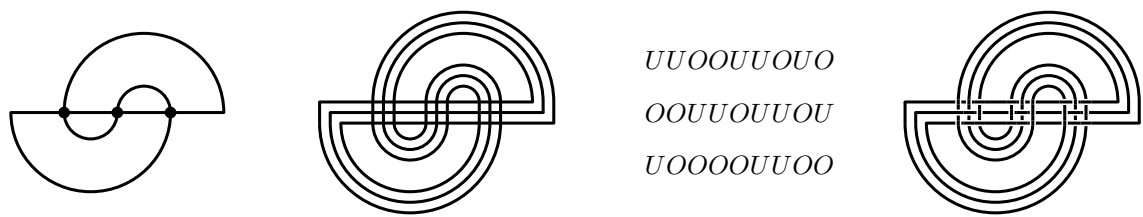
\begin{figure}[h]
\begin{center}
\begin{tikzpicture}[scale = .7]

\def \w {.4}

\begin{scope}[xshift = 0cm]

\draw [line width = \w mm, line cap=round] (0,0) to (4,0);

\draw [line width = \w mm] (3,0) arc [radius=.5, start angle=0, end angle= 180];
\draw [line width = \w mm] (4,0) arc [radius=1.5, start angle=0, end angle= 180];
\draw [line width = \w mm] (3,0) arc [radius=1.5, start angle=0, end angle= -180];
\draw [line width = \w mm] (2,0) arc [radius=.5, start angle=0, end angle= -180];

\foreach \a in {1,2,3}{
\draw [fill] (\a,0) circle [radius=0.1];
}

\end{scope}

\begin{scope}[xshift = 6cm]
\foreach \a in {0,.2,-.2}{
\draw [line width = \w mm, line cap=round] (-\a,\a) to (4-\a,\a);
}

\draw [line width = \w mm] (-.2,.2) to (-.2,-.2);
\draw [line width = \w mm] (0,0) to (0,-.2);

\draw [line width = \w mm] (4.2,.2) to (4.2,-.2);
\draw [line width = \w mm] (4,0) to (4,.2);

\foreach \b in {1,...,3}{
\foreach \a in {0,.2,-.2}{
\draw [line width = \w mm] (\a+\b,.2) to (\b+\a,-.2);
}
}

\foreach \a in {0,-.2,.2}{
\foreach \b in {.2}{
\draw [line width = \w mm] (3+\a,\b) arc [radius=.5+\a, start angle=0, end angle= 180];
\draw [line width = \w mm] (4+\a,\b) arc [radius=1.5+\a, start angle=0, end angle= 180];
}
\foreach \b in {-.2}{
\draw [line width = \w mm] (3+\a,\b) arc [radius=1.5+\a, start angle=0, end angle= -180];
\draw [line width = \w mm] (2+\a,\b) arc [radius=.5+\a, start angle=0, end angle= -180];
}
}

\end{scope}

\begin{scope}[xshift = 17cm]

\node [] at (-3.3,1) {\Small{$UUOOUUOUO$}};
\node [] at (-3.3,0) {\Small{$OOUUOUUOU$}};
\node [] at (-3.3,-1) {\Small{$UOOOOUUOO$}};

\foreach \a in {0,.2,-.2}{
\draw [line width = \w mm, line cap=round] (-\a,\a) to (4-\a,\a);
}

\draw [line width = \w mm] (-.2,.2) to (-.2,-.2);
\draw [line width = \w mm] (0,0) to (0,-.2);

\draw [line width = \w mm] (4.2,.2) to (4.2,-.2);
\draw [line width = \w mm] (4,0) to (4,.2);

\foreach \b in {1,...,3}{
\foreach \a in {0,.2,-.2}{
\draw [line width = \w mm] (\a+\b,.2) to (\b+\a,-.2);
}
}

\foreach \a in {0,-.2,.2}{
\foreach \b in {.2}{
\draw [line width = \w mm] (3+\a,\b) arc [radius=.5+\a, start angle=0, end angle= 180];
\draw [line width = \w mm] (4+\a,\b) arc [radius=1.5+\a, start angle=0, end angle= 180];
}
\foreach \b in {-.2}{
\draw [line width = \w mm] (3+\a,\b) arc [radius=1.5+\a, start angle=0, end angle= -180];
\draw [line width = \w mm] (2+\a,\b) arc [radius=.5+\a, start angle=0, end angle= -180];
}
}

\foreach \a in {1.2,1.8,2.8,3.2}{
\draw [line width = 2*\w mm, white] (\a - .1,.2) to (\a + .1,.2);
\draw [line width = \w mm, black] (\a - .1,.2) to (\a + .1,.2);
}

\foreach \a in {.8,1,2,2.2,3}{
\draw [line width = 2*\w mm, white] (\a,.25) to (\a,.15);
\draw [line width = \w mm] (\a,.25) to (\a,.15);
}

\foreach \a in {.8,1,2,3}{
\draw [line width = 2*\w mm, white] (\a - .1,0) to (\a + .1,0);
\draw [line width = \w mm, black] (\a - .1,0) to (\a + .1,0);
}

\foreach \a in {1.2,1.8,2.2,2.8,3.2}{
\draw [line width = 2*\w mm, white] (\a,.05) to (\a,-.05);
\draw [line width = \w mm] (\a,.05) to (\a,-.05);
}

\foreach \a in {1,1.2,1.8,2,3,3.2}{
\draw [line width = 2*\w mm, white] (\a - .1,-.2) to (\a + .1,-.2);
\draw [line width = \w mm, black] (\a - .1,-.2) to (\a + .1,-.2);
}
\foreach \a in {.8,2.2,2.8}{
\draw [line width = 2*\w mm, white] (\a,-.25) to (\a,-.15);
\draw [line width = \w mm] (\a,-.25) to (\a,-.15);
}

\end{scope}

\end{tikzpicture}
\caption{From left to right: an example of a graph $\Gamma_3$, a graph $\Gamma_3^3$, a crossing information string $V^3_3$, and the knot obtained as $(\Gamma_3^3, V^3_3)$.}
\label{fig:3random}
\end{center}
\end{figure}

The {\em $(r,{2s-1})$-meander link diagram or link} is the link diagram or the link  obtained by adding the crossing information string $V^r_{2s-1}$ to the graph $\Gamma^r_{2s-1}$ as above. We denote the set of all  $(r,{2s-1})$-meander link diagrams for a given $s$ and $r$ by $\mathcal{L}^r_{2s-1}$. Note that one link can be in several such sets for different $r, s$.

An example of a meander graph, an $(r,{2s-1})$-meander graph, a crossing information string, and the $(r,{2s-1})$-meander link diagram obtained from them is given in Figure \ref{fig:3random}.

A planar isotopy applied to a $(r,{2s-1})$-meander link diagram transforms it into a diagram called a \textit{potholder diagram} in \cite{EZHLN2}. We therefore can reformulate Theorem 1.4 from \cite{EZHLN2} as follows.

\begin{theorem}\label{links} Any link has an $(r,{2s-1})$-meander link diagram.
\end{theorem}

Theorem \ref{links} means, in particular, that the set of all links is the same as the set of $(r,{2s-1})$-meander links  for all natural $r$ and $s$, up to link isotopy. If we set $r=1$, the set $\bigcup_{s=1}^\infty\mathcal{L}^1_{2s-1}$ contains the set of all knots by Theorem \ref{thm:AST}. It also contains some links with more than one component: see Figure \ref{fig:parenLink} for an example of a link in this set.

For positive integers $s$ and $r$, choose $\Gamma^r_{2s-1}$ and $V^r_{2s-1}$ uniformly at random from the set of all $(r,{2s-1})$-meander graphs and the set  all $(r,{2s-1})$-crossing information strings. The resulting $(r,{2s-1})$-meander link diagram or link is called a \textit{random $(r,{2s-1})$-meander link diagram or link}.

\section{Links with pierced circles}\label{SecCircles}

Here and further we consider link complements in $S^3$. When projected to a plane or $S^2$, certain fragments of a link diagram guarantee the the link is not trivial, i.e. is not an unlink in $S^3$. In meander link diagrams, we identify one such fragment below. Later in this section, we will find the mathematical expectation of such a fragment being present.

A \textit{pierced circle} is a fragment of a meander graph or link depicted in Figure \ref{fig:pierced}. The circle component has exactly two consecutive vertices or crossings on the axis. See Figure \ref{fig:parenLink}, middle and right, for an example of a meander graph and a meander link with a pierced circle.

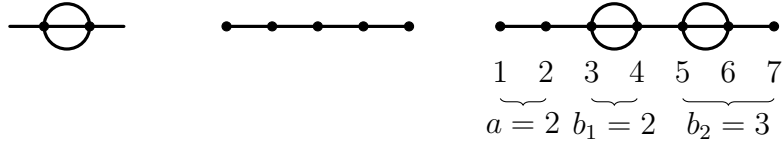
\begin{figure}[htbp]
\begin{center}
\begin{tikzpicture}[scale = .6]

\draw [very thick, line cap=round] (0.25,0) to (2.75,0);
\draw [very thick] (1,0) arc [radius=.5, start angle=-180, end angle= 180];

\foreach \a in {1,2, 5,6,7,8,9,11,12,13,14,15,16,17}{
\draw [fill] (\a,0) circle [radius=0.1];
}

\draw [very thick, line cap=round] (5,0) to (9,0);

\draw [very thick, line cap=round] (11,0) to (17,0);
\draw [very thick] (13,0) arc [radius=.5, start angle=-180, end angle= 180];
\draw [very thick] (15,0) arc [radius=.5, start angle=-180, end angle= 180];

\foreach \a in {1,...,7}{
\node at (10+\a,-1) {\a};
}

\draw[decoration={brace,mirror,raise=5pt},decorate]
  (11,-1.3) -- node[below=6pt] {$a=2$} (12,-1.3);

\draw[decoration={brace,mirror,raise=5pt},decorate]
  (13,-1.3) -- node[below=6pt] {$b_1=2$} (14,-1.3);
  
\draw[decoration={brace,mirror,raise=5pt},decorate]
  (15,-1.3) -- node[below=6pt] {$b_2=3$} (17,-1.3);

\end{tikzpicture}
\caption{From left to right: a pierced circle in a graph; the path graph $P_5$; a graph from $\mathcal{A}_7(2)$ on the right.}
\label{fig:pierced}
\end{center}
\end{figure}

We will determine the expected number of pierced circles in a random meander diagram.

Consider the path graph $P_v$, as in Figure \ref{fig:pierced}, middle.  Label the vertices 1 through $v$ from left to right. A pierced circle is \textit{at position} $i$ if the left vertex of the circle is labeled $i$, $1\leq i \leq v-1$. We will count the number of ways to add $k$ pierced circles to $P_v$ so that no two circles share a vertex. Let $\mathcal{A}_v(k)$ be the collection of the resulting graphs, with $k$ pierced circles added.

Recall that the number of integer ordered partititons (or compositions) of a nonegative integer $n$ into $k$ natural numbers is the number of ways to write $n$ as $a_1+a_2+\cdots +a_k$ for some natural numbers $a_1, a_2, \ldots, a_k \geq 1$. The quantity $|\mathcal{A}_v(k)|$ can be seen as the number of integer partititons of $v$ into one non-negative integer, say $a$, and $k$ integers that are two or greater, say $b_1, b_2, ..., b_k$, where $b_i\geq 2$, as follows. The non-negative integer $a$ represents the number of vertices before the position of the first pierced circle, which may be zero.  Now label $k$ pierced circles in $P_v$ by $1, 2, ..., k$ from left to right. For $m=1, ..., k-1,$ the integer $b_m$ is determined by the pair of $m$-th and $(m+1)$-th consecutive circles in the graph $P_v$, say located at positions $i$ and $j$, with $b_m=\vert i-j \vert$. And the last integer $b_k$ is the number of vertices from the position of the last circle, say position $q$, to the end of the path graph, as $v-q+1$. In Figure \ref{fig:pierced} with $v=7$ and $k=2$, we have $a=2, b_1=2, b_2=3$, i.e. $7 = 2 + 2 + 3$.

\begin{lemma}\label{lem:numKcircles}
There are ${v-k \choose k}$ graphs in the set $\mathcal{A}_v(k)$. 
\end{lemma}

\begin{proof}

Above, we observed that $|\mathcal{A}_v(k)|$ is the number of ordered partitions of $v$ into  $a, b_1, b_2, ..., b_k$ where $a\geq 0$ and $b_i\geq 2$.  We relabel the summands by taking a bijective function $f$ that maps $a\mapsto a_1-1$ and $b_i \mapsto a_{i+1} +1$ for $i = 1,2, \ldots k$.  Then $a_i\geq 1$ for all $i$. Using this relabeling, we have

\begin{equation*} 
\begin{split}
v & = (a_1-1) + (a_2+1)+(a_3+1) +\cdots +(a_{k+1}+1)\\
v & = a_1+a_2+\cdots +a_{k+1} +k-1\\
v-k+1& = a_1+a_2+\cdots +a_{k+1}.
\end{split}
\end{equation*}
This an ordered partition of $v-k+1$ into $k+1$ natural numbers. But every such ordered partititon corresponds to an ordered partition of $v$ into $a, b_1, b_2, ..., b_k$, through the bijection $f$. Hence the number of such ordered partititons is the same as the number of ordered partitions of $v$ into  $a, b_1, b_2, ..., b_k$, and is equal to  $|\mathcal{A}_v(k)|$. In general, the number of ordered partititons of a nonegative integer $n$ into $j$ natural numbers is ${ n-1 \choose j-1}$ \cite{Stanley}. Therefore $|\mathcal{A}_v(k)|={v-k \choose k}$.

\end{proof}

Recall that the \textit{Catalan number} $C_s = \frac{1}{s+1}{2s \choose s}$ counts the number of valid strings of $s$ pairs of parentheses \cite{Stanley}, i.e. the number of $p$-strings of length $s$.

A $p$-string has $s$ pairs of parentheses, so $2s$ positions for a parenthesis, which is either left or right.  We call a matched pair of parentheses next to each other, \texttt{()}, a {\em nesting}. A nesting is \textit{at position} $i$ if the left parenthesis is at the $i$-th position in the $p$-string.  For example, consider the $p$-string of length 6:  \texttt{(())((()()))}. Replace the parentheses that are not a part of any nesting with a dash: \texttt{-()---()()--}.  We see that the positions of three nestings in this $p$-string are 2, 7, and 9.

\begin{lemma}\label{lem:Csminusj} Fix a $p$-string $P$ of length $s, s\geq 2,$ with $j$ nestings, $j\geq0$. Vary the parentheses in $P$ that do no belong to the nestings. The number of resulting $p$-strings is $C_{s-j}$. \end{lemma}

Note that each resulting $p$-string in the lemma contains nestings at the same positions as $P$, but also possibly other nestings.

\begin{proof} Let the positions for the nestings in $P$ be $x_1, x_2, ..., x_j$. Then there are $2(s-j)$ parentheses that are not part of a nesting in $P$. We replace each of them with a dash, as above. 

Now choose a $p$-string $P'$ of length $s-j$. Fill in the dashes of $P$ by substituting the parentheses of $P'$, in the same order as in $P'$. This gives a bijection between strings of length $s-j$ (like $P'$), and the strings we are counting in the lemma statement.  And there are $C_{s-j}$ distinct $p$-strings of length $s-j$. \end{proof}

Recall that a meander graph $\Gamma^1_{2s-1}$ has no added parallel strands, and exactly ${2s-1}$ vertices of valence 4. Let the set of such graphs that have exactly $k$ pierced circles be denoted by $\mathcal{E}_s(k)$, and let $|\mathcal{E}_s(k)|=E(s, k)$.

Let $\mathcal{O}_s(k)$ be a set of ordered triples 
$$\mathcal{O}_s(k) = \{ (A, P, Q)  |  A\in\mathcal{A}_{2s-1}(k), \text{ and } P,Q \text{ are $p$-strings of length } s-k  \}.$$ 
Denote $|\mathcal{O}_s(k)|=O(s, k)$.

We now establish an upper bound for the number $E(s, k)$

\begin{lemma}\label{newEandO}Given $s\geq 1$ and $k$ for $0\leq k\leq s$,
\begin{enumerate}
\item $O(s, k) =  {{2s-k-1 } \choose {k} } (C_{s-k})^2$\\
\item $O(s, k) =  \sum_{m=k}^s \ch{m}{k} E(s, m)$\\
\end{enumerate}
\end{lemma}

\begin{proof} For part (1), recall that $\mathcal{A}_{2s-1}(k)$ denotes the set of path graphs with $2s-1$ vertices and $k$ pierced circles. The number of such graphs $|\mathcal{A}_{2s-1}(k)|$ is given by Lemma \ref{lem:numKcircles}, once we substitute $v=2s-1$ in the lemma statement, and it is $|\mathcal{A}_{2s-1}(k)|={{2s-k-1 } \choose {k} }$.  Recall also that $\mathcal{O}_s(k)$ is the Cartesian product of $\mathcal{A}_{2s-1}(k)$, the set $\mathcal{P}$ of $p$-strings $P$ of length $s-k$, and one more such set $\mathcal{Q}$. Since $|\mathcal{P}|=|\mathcal{Q}|=C_{s-k}$ by Lemma \ref{lem:Csminusj}, we obtain:
\begin{equation}\label{EandO}
O(s, k) = |\mathcal{O}_s(k)|= {{2s-k-1 } \choose {k} } (C_{s-k})^2.
\end{equation}

For part (2), consider a meander graph $\Gamma = \Gamma^1_{2s-1}$ with exactly $j\geq k$ pierced circles. We claim that $O(s, k)$ counts $\Gamma$ exactly $ j \choose {k}$ times. Indeed, take a subgraph $A$ of $\Gamma$ by leaving $k$ of the $j$ pierced circles of $\Gamma$, and the axis, and deleting all other edges of $\Gamma$. The graph $A$ is an element of $ \mathcal{A}_s(k)$. There are exactly $j \choose k$ such distinct subgraphs of $\Gamma$. For each such subgraph $A$, there is exactly one pair of $p$-strings of length $s-k$ that, combined with $A$, yields $\Gamma$.  Thus there are $\ch{j}{k}$ elements of $\mathcal{O}_s(k)$ that all correspond to $\Gamma$. But the graphs like $\Gamma$, with $k$ pierced circles, make up the set $\mathcal{E}_s(k)$, with $|\mathcal{E}_s(k)|=E(s, k)$ by definition. Then for each $m$ with $k\leq m\leq s$, we sum all occurences of each such graph $\Gamma$, and obtain  
\begin{center}

$O(s, k) = E(s,k) + \ch{k+1}{k}E(s,k+1) +\cdots+ \ch{s}{k}E(s,s) = \sum_{m=k}^s \ch{m}{k} E(s, m)$.

\end{center}
 \end{proof}  


Note that Lemma \ref{newEandO}(2) implies that $O(s, k) \geq  E(s,k)$, with equality when $s=k$. We will also obtain a new relation between these quantities in the next lemma, that will be useful for us later.

The following lemma will be useful in the next section.

\begin{lemma}\label{Circles}

Given $s\geq 1$ and $k$ for $0\leq k\leq s$,
$$E(s, k) = \sum_{m=k}^s (-1)^{m+k}\ch{m}{k}O(s, m).$$

\end{lemma}

\begin{proof}

Rearrange Lemma \ref{newEandO}(2) as follows:

\begin{equation}\label{Ek}
E(s, k)=O(s, k) - \sum_{m=k+1}^s \ch{m}{k} E(s, m).
\end{equation}

Apply strong induction on $s-k$ for a fixed $s$, that is, we fix $s$ and let $k$ decrease: $k=s, s-1, s-2, ..., 1, 0$. As noted before the lemma, $O(s, s) = E(s, s) = 0$. This is the base step.

Assuming the proposition statement holds for $ E(s,s), E(s,s-1), \ldots, E(s,k+1)$, we prove it for $E(s,k)$. Plugging in the formula from the proposition statement into equation (\ref{Ek}), we obtain

\begin{center}

$E(s, k)   = O(s, k) - \sum_{m=k+1}^s   \ch{m}{k} \sum_{j=m}^s (-1)^{m+j}\ch{j}{m}O(s, j)$.

\end{center}

Via rearranging and regrouping, we then have the following:

\begin{equation}\label{star}
E(s, k)=O(s, k) - \sum_{m=k+1}^s O(s, m) \ch{m}{k}   \left[ \sum_{i=k+1}^{m} (-1)^{m-i} \ch{m-k}{m-i}\right].\tag{$*$}
\end{equation}

The rearranging above is trivial, and we give full details, with sum expansions, in Appendix, as Claim \ref{claim}.

In the above formula, let $t=m-i$. Recall that the sum of alternating binomial coefficients is zero, i.e.  $ \sum_{t = 0}^{m-k} (-1)^{t} \ch{m-k}{t} = 0$.  Thus the inner sum on the right in (\ref{star}) becomes

$\sum_{t = 0}^{m-k-1} (-1)^{t} \ch{m-k}{t} =
\left[ \sum_{t = 0}^{m-k} (-1)^{t} \ch{m-k}{t} \right] -(-1)^{m-k} \ch{m-k}{m-k}  = (-1)^{m-k+1}$.

With this, (\ref{star}) is now

\begin{center}
$E(s, k)=O(s, k) - \sum_{m=k+1}^s \ch{m}{k}O(s, m)\left[(-1)^{m-k+1}\right]=\sum_{m=k}^s (-1)^{m+k}\ch{m}{k}O(s, m).$
\end{center}

	 \end{proof}

\begin{proposition}\label{circles}
For a fixed $s$, the expected number of pierced circles in $\Gamma^1_{2s-1}$ is $\mathbb{E} = \frac{ O(s, 1) }{C_s^2}.$
\end{proposition}

\begin{proof}
For a random graph $\Gamma^1_{2s-1}$, introduce an indicator random variable $X_i$ that takes the value $1$ when a pierced circle appears at a position $i$ in $\Gamma^1_{2s-1}$, and value 0 otherwise. There are $2s-2$ possible positions for the pierced circle in a graph $\Gamma^1_{2s-1}$ by construction. Note that the variables $X_1, X_2, ... X_{2s-2}$ are not independent: e.g. if a pierced circle appearred at a position 1, it cannot appear at a position 2 in the same graph. 

We claim that the probability $P(X_i=1)=(C_{s-1}/C_s)^2$. Indeed, the graph $\Gamma^1_{2s-1}$ corresponds to two $p$-strings, and each string must have at least one nesting, at the position $i$. By Lemma \ref{lem:Csminusj}, the number of options for two distinct $p$-strings, with one nesting each, is $C_{s-1}^2$, and the number of options for any two distinct $p$-strings is  $C_s^2$.

Now by the linearity of expectation, the expectation of the number of pierced circles in  $\Gamma^1_{2s-1}$ is $\mathbb{E}=P(X_1=1)+P(X_2=1)+...+P(X_{2s-2}=1)=(2s-2)(C_{s-1}/C_s)^2$. According to formula (\ref{EandO}), $(2s-2)C^2_{s-1}=O(s, 1)$, and the expectation is $\mathbb{E}=\frac{ O(s, 1) }{C_s^2}$.

\end{proof}

From a straightforward simplification of the definitions, the expectation has the nice closed form, $  \frac{ O(s, 1) }{C_s^2} = \frac{s^3+s^2-s-1}{8s^2-8s+2}$.   This asymptotically approaches $ \frac{s+2}{8}$ from above for large $s$.  In fact, by $s=6$, the error between $ \frac{s+2}{8}$ and $ \frac{ O(s, 1) }{C_s^2}$ is less than $0.013$, and we expect about one pierced circle.

\section{Unlinks are rare}\label{Unlinks}

A natural question to ask is whether the new model is likely to produce a non-trivial link.  Our goal now is to show that the random links generated in our model are nontrivial.  We will do this by first showing that as $s \to \infty$ for a link $\Gamma_{2s-1}^1$, the number of corresponding link diagrams with no pierced circles is small compared to the total number of diagrams.

A hypergeometric series is a series of the form $\sum t_k$, where the ratio of consecutive terms $\frac{t_{k+1}}{t_k}$ is a rational function of $k$.   To prove Proposition \ref{thm:circlesPresent}, we will consider such series $\frac{E(s, 0)}{C_s^2}$, which is the ratio of the number of meander graphs with no pierced circles to all meander graphs for fixed $s$. We will obtain a recurrence relation for $E(s, 0)$.  Once we have a recurrence relation, we can apply the following classic result of Poincar\'e to show that this ratio becomes small, and therefore almost all meander graphs have a pierced circle. 

Suppose $\{u_n\}$ is a number sequence for natural indices $n$. Consider a linear recurrence relation in $k+1$ terms of this sequence for some fixed positive natural $k$, i.e. $\alpha_{0, n}u_{n+k} +\alpha_{1,n}u_{n+k-1} +\alpha_{2,n}u_{n+k-2} +\cdots +\alpha_{k,n}u_{n}+c =0$. Here $c$ and $\alpha_{i,n} \in \mathbb{R}$ for $i=0, 1, ..., k$ are coefficients, and the relation holds for all natural $n$. We can assume $\alpha_{0, n}=1$. In \cite{Poincare}, such a relation is called a difference equation. It is called homogeneous if the constant term $c=0$. We will use the following theorem.

 In what follows, $u=(u_n, ...., u_{n+k})$.

\begin{theorem}\label{thm:poincare}{ [{Poincar\'e} \cite{Poincare}]}
Suppose the coefficients $\alpha_{i,n}$ for $i = 1, 2, \ldots, k$, of a linear homogeneous difference equation 
\begin{equation}\label{eq:poincare}
u_{n+k} +\alpha_{1,n}u_{n+k-1} +\alpha_{2,n}u_{n+k-2} +\cdots +\alpha_{k,n}u_{n} =0
\end{equation}
\noindent have limits $\lim_{n\to \infty} \alpha_{i,n}= \alpha_i$ for $i = 1, 2, \ldots, k$, and the roots $\lambda_1, \ldots, \lambda_k$ of the characteristic equation $t^k+\alpha_1t^{k-1}+\cdots + \alpha_k = 0$ have distinct absolute values. Then for any solution $u$ of equation (\ref{eq:poincare}), either $u_n = 0 $ for all sufficiently large $n$ or $\lim_{n\to\infty} \frac{u_{n+1}}{u_n} = \lambda_i$ for some $i$.
\end{theorem}

\begin{lemma}\label{lem:Zeilbergers}
The sequence $\{E(s, 0) \}_{s=1}^\infty$ satisfies the recurrence relation $$\sum_{k=0}^3 P_k(s) E(s+k, 0) = 0,$$ with polynomials $P_k$ given by

\begin{table}[ht]
\begin{center}
\begin{tabular}{ l c l }
$P_0(s) = 2 s^3+s^2-8 s+5$, & &$P_1(s) = -26 s^3-93 s^2-82 s-30$,\\
$P_2(s) = -26 s^3-141 s^2-226 s-81$,  & & $P_3(s)  = 2 s^3+17 s^2+40 s+16$.
\end{tabular}
\end{center}
\end{table}%

\end{lemma}

\begin{proof}
This was verified using the computer algebra software {Maple}\texttrademark\ \cite{Maple} by applying command \texttt{Zeilberger} from subpackage \texttt{SumTools[Hypergeometric]}.

For a short explanation of Zeilberger's algorithm, recall that we are looking for a recurrence relation on $E(s, 0)$. By Lemma \ref{Circles}, $E(s, 0)=O(s, 0) - O(s, 1)  +O(s, 2)-\cdots +(-1)^s O(s, s)$. Recall also that $O(s, k) =  {{2s-k-1 } \choose {k} } (C_{s-k})^2$ by formula (\ref{EandO}). The polynomials $P_0, P_1, P_2, P_3$ are obtained and verified using Zeilberger's Algorithm, described, for example, in \cite{aeqb}. In particular, the command \texttt{Zeilberger} returns these polynomials and a certificate verifying the calculation. We will outline the mathematics behind the command \texttt{Zeilberger} and the verification, to show that this leads to a rigorous proof.

Using the {Maple}\texttrademark\   notation, the command  \texttt{Zeilberger}  takes as input the function $T(s,m) = (-1)^m O(s, m)$, so that $E(s, 0)=\sum_{m=0}^s T(s,m) $. The command then outputs the functions $L$ and $G$, which we describe later. By  $x = x(s)$ the {\em shift} operator is denoted, that is, a function $x$ of $s$ such that $xT(s,m) = T(s+1,m)$. In our notation, $x(-1)^m O(s, m) = (-1)^m O(s+1, m)$.  One can solve for $x$ without difficulty, but we only need the existence of such $x$ for the proof.  

It is also important for the final simplification below that we can extend the definition of $E(s, 0)$ as follows.  The sum $E(s, 0)$ has summands that range from $m=0$ to $s$. We can increase the number of summands by increasing the largest value of $m$. This will not change $E(s, 0)$, since the terms $O(s, m)$ are zero if $m\geq s$.  Thus, in the remaining part of the proof we may consider $E(s, 0), E(s+1, 0), E(s+2, 0)$, and $E(s+3, 0)$ to have summands up to $m=s+3$.  Then we obtain, using the shift operator $x$,

\begin{center}
$\sum_{k=0}^3 P_k(s) E(s+k, 0) =  \sum_{k=0}^3 \left( P_k(s) \sum_{m=0}^{s+3}(-1)^m O(s+k, m) \right) = $

$=   \sum_{m=0}^{s+3}\left( \sum_{k=0}^3  P_k(s)(-1)^m O(s+k, m) \right) =   \sum_{m=0}^{s+3}\left( \sum_{k=0}^3  x^kP_k(s)(-1)^m O(s, m) \right) =$

$=   \sum_{m=0}^{s+3}\left( (-1)^m O(s, m) \sum_{k=0}^3  x^kP_k(s)\right)$ 
\end{center}

One part of the output of \texttt{Zeilberger} command is $L = P_0(s) + xP_1(s) +x^2P_2(s) + x^3P_3(s) $ where $P_k$ are the four explicit polynomials Zeilberger's algorithm provides us. The other part of the output is $G$, which denotes a function $G(s,m)$ such that $L(T(s,m) ) = G(s,m+1) - G(s,m)$. We exclude the explicit expression for $G$ for brevity, but it is provided by Maple. See \cite{MapleFile} for the explicit expression and the details of the computation. The function $G$ is what allows us to creatively telescope.  In particular, for $0\leq m \leq s+3$,

\begin{center}

$\sum_{m=0}^{s+3}\left( (-1)^m O(s, m) \sum_{k=0}^3  x^kP_k(s)\right)  =   \sum_{m=0}^{s+3}\left( T(s,m) L \right) =$

$= \sum_{m=0}^{s+3}\left( G(s,m+1)-G(s,m)\right) =    G(s,s+4)-G(s,0) = 0$,
\end{center}

where the last equality follows from inspecting the explicit expression for $G$.  \end{proof}

Now we have all the pieces for the next theorem.

\begin{proposition}\label{thm:circlesPresent}
The ratio of $(1, 2s-1)$-meander graphs without any pierced circles to all $(1, 2s-1)$-meander graphs in the random meander model asymptotically approaches zero, i.e. $\frac{E(s, 0)}{C_s^2} \to 0$ as $s\to \infty$.
\end{proposition}

\begin{proof} Let $a_s = \frac{ E(s, 0)  }{C_s^2}$.  We will use the ratio test for sequences and show that $\lim_{s\to\infty} \frac{a_{s+1}}{a_{s}}<1$.  It is straight forward to see that $\lim_{s\to\infty} \frac{C_{s}}{C_{s+1}} = \lim_{s\to\infty} \frac{s+2}{2(2s+1)}=\frac{1}{4}$, and notice that
$$\frac{a_{s+1}}{a_{s}} = \frac{E(s+1, 0)  {C_s^2}}{E(s, 0)  C_{s+1}^2}.$$

This allows us to reduce our problem to showing that
\begin{equation}\label{eq:ratio}
\lim_{s\to\infty} \frac{E(s+1, 0)}{E(s, 0)} < 16.
\end{equation}

We divide the recurrence relation from  Lemma \ref{lem:Zeilbergers} by $P_3(s)$. 

$$ E(s+3, 0)+ \frac{P_2(s)}{P_3(s)}E(s+2, 0)+\frac{P_1(s)}{P_3(s)}E(s+1, 0)+\frac{P_0(s)}{P_3(s)}E(s, 0) = 0 .$$

The above is a linear homogeneous difference equation in $u_s= E(s, 0) $ with 4 terms. We can apply Poincar\'e's theorem to it, with $\alpha_{1,s}=\frac{P_{2}(s)}{P_3(s)} , \alpha_{2,s}=\frac{P_{1}(s)}{P_3(s)}$, and $\alpha_{3,s}=\frac{P_{0}(s)}{P_3(s)}$.  Then, as defined in the statement of Poincar\'e's theorem, $ \alpha_i = \lim_{s\to\infty} \alpha_{i,s} $ and via basic calculus, we see $\alpha_1 = -13, \alpha_2 = -13, $ and $\alpha_3 = 1$.  Thus, the characteristic equation of this recurrence is $t^3-13t^2-13t+1=0$.

 There are three real solutions to the characteristic equation, $z=-1, 7-4\sqrt{3}\approx 0.0718,$ and $7+4\sqrt{3} \approx 13.928$, all of which are less than 16.  Theorem \ref{thm:poincare} states that $\lim_{s\to \infty} \frac{E(s+1, 0)}{E(s, 0)}$ converges to zero or one of the solutions of the characteristic equation, which verifies the inequality (\ref{eq:ratio}), finishing the proof. \end{proof}

Consider a single circle $C$ in a diagram, with $r$ horizontal strands crossing it, {as in} Figure \ref{fig:unlinks}.  In our model, the $r$ horizontal strands are parallel copies of a knot. The crossing information is chosen at random.

\begin{figure}[htbp]
\begin{center}
\begin{tikzpicture}[scale = .5]

\foreach \a in {.5}{
\foreach \b in {1.3}{
\draw [line width = \a mm, line cap=round] (0,.75) to (4,.75);
\draw [line width =  \a mm, line cap=round] (0,-.75) to (2,-.75);

\draw [line width =  \b mm, white] (3.5,0) arc [radius=1.5, start angle=0, end angle= 360];
\draw [line width =  \a mm] (3.5,0) arc [radius=1.5, start angle=0, end angle= 360];

\draw [line width =  \b mm, line cap=round, white] (0,0) to (4,0);
\draw [line width =  \a mm, line cap=round] (0,0) to (4,0);

\draw [line width =  \b mm ,white] (2,-.75) to (4,-.75);
\draw [line width =  \a mm, line cap=round] (2,-.75) to (4,-.75);

\node at (3.25,2) {$C$};

}
}

\end{tikzpicture}
\caption{The circle $C$ meeting three parallel copies of the axis {component}. The top two are {not linked with other link components shown, while} the bottom one is linked {with $C$}. }
\label{fig:unlinks}
\end{center}
\end{figure}
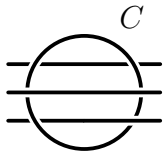

\begin{lemma}\label{lem:srlink}
The component $C$, which is a circle in a meander link diagram, is unlinked from other components in the $(s,r)$-random link with probability $\frac{1}{2^r}$.
\end{lemma}

\begin{proof} For an axis component, consider its two crossings with $C$.  There are four options for the crossing information string for such a pair of crossings: (over, over), (over, under), (under, over),  and (under, under). Note that $C$ is unlinked from an axis component if the axis component passes under $C$ at both its crossings with $C$ or passes over $C$ at both (see Figure \ref{fig:unlinks}).  Hence half of the options for the crossing information imply that $C$ and the axis component are not linked. Therefore, if there are $r$ axis components, the probability that all of them are unlinked from $C$ is $\frac{1}{2^r}$.
\end{proof}

The random model described above is easily modified to produce only links that are alternating. Each $(r, 2s-1)$-meander graph yields the same number of alternating link diagrams. Indeed, every link projection can be made into an alternating link diagram with the correct choice of under and overpasses (i.e. of the crossing information). Therefore, this is also true for any $(r, 2s-1)$-meander graph. Moreover, there are exactly two options for alternating crossing information once a meander graph is given: choose an arbitrary graph vertex $v$, and connect two horizontal edges adjacent to $v$ by an underpass at $v$. This determines the first alternating link. The second link is obtained by making an overpass between horizontal edges at $v$. Such a (restricted) model that takes any meander graph together with a choice of alternating crossing information is what we will refer to whenever we discuss random alternating links.

We can now state the main result of this section. 

\begin{theorem}\label{thm:Rare}
Let $L$ be a $(r, 2s-1)$-random meander link.  If $L$ is alternating and $s$ tends to infinity, then $L$ is nontrivial with probability one.  If $L$ has its crossings chosen at random and  $s$ and $ r$ tend to infinity, then  $L$ is nontrivial with probability one.
\end{theorem}

\begin{proof} The first statement follows from Proposition \ref{thm:circlesPresent}, since an alternating link with a pierced circle cannot be an unlink. 
Similarly in the second statement, we are guaranteed there is a pierced circle as $s\to\infty$. But with crossing information string chosen at random, the circle might be unlinked from other link components. Letting $r\to\infty$, Lemma \ref{lem:srlink} ensures at least one pierced circle will be linked to an axis component. 
\end{proof}

\subsection{More on distribution of links in random meander model.} We conclude this section with several further observations and questions about obtaining different links and knots in this model.

 The above theorem shows that the probability of obtaining an unlink in this random model approaches zero as $r, s$ grow. This is in fact true if we substitute the unlink by another fixed link. 
For a link $L$, denote its number of components by $t_L$. Note that if $L$ is an $(r, 2s-1)$-random link, then $r\leq t_L \leq sr$.

\begin{proposition}\label{nonisotopic}
For a given link $K$ and a $(r, 2s-1)$-random link $L$, the probability that $K$ and $L$ are isotopic $P(L=K) \to 0$ as $r\to \infty$. 
\end{proposition}

\begin{proof}
As $t_K$ is finite, there is some $r > t_K$, and every  $(r, 2s-1)$-random link $L$ has more components than the chosen link $K$.
\end{proof}

The above ensures that we can produce infinitely many distinct or non-isotopic links using the random meander link model. It is also interesting how many distinct links are created by using this model for every $r, s$. This can be answered for knots, which only occur when $r=1$.  

\begin{proposition}\label{prime}
There are at least $\Omega({2.68}^n)$ distinct meander prime knots with $(2^{n+1} -12)$ crossings. 
\end{proposition}

\begin{proof}
Suppose we have a knot $K$ with crossing number $c(K)$. Denote the number of crossings in $(1, 2s-1)$-meander diagram of $K$ by $m(K)$. Theorem \ref{thm:OwadMeander} says that every knot has a meander diagram. Moreover, for random meander knots, Theorem 3.2 and Proposition 3.4 in \cite{Owad2} combined yield $m(K)\leq 4(2^{c(K) -1} - 1) -8 = 4(2)^{c(K) -1} -12 = 2^{c(K)+1} -12$. Therefore, once we generate meander knot diagrams with at most $2^{n+1} -12$ crossings, we also generate all knots with $n$ crossings. The number of prime knots with $n$ crossings was given by Welsch \cite{Welsh} and is at least $\Omega (2.68^n)$.
\end{proof}

To make a similar observation for links ($r>1$), a bound on the crossing number of a meander link is needed in terms of the crossing number of a link. Then the number of prime links with $n$ crossings can be used similarly, which can also be found in Welsch \cite{Welsh}. 

\begin{question} Given an $(r, 2s-1)$-meander link $L$ with crossing number $m(L)$, what is an upper bound on the minimal crossing number $c(L)$ of this link?
\end{question}

\section{Expected number of twists and Expected volume}\label{VolumeEtc}

In this section, we investigate some properties of links in our model. We find the expected number of twists for our diagram. We give bounds on the mathematical expectation for hyperbolic and simplicial volume of link complements. 

\subsection{Expected number of twists in a random link diagram.} In Section \ref{SecCircles}, we discussed how Catalan numbers are related to strings of parentheses. Consider \texttt{()}, an open parenthesis which is adjacent to its matching closed parenthesis in a string. Now let us count the number of $p$-string of length $n$ with exactly $k$ such substrings \texttt{()}. The number is the Narayana number, denoted by $N(n,k)$.  By definition, $\sum_{k=1}^n N(n,k) = C_n$. Explicitly, $N(n,k) = \frac{1}{n} {{n}\choose k}{{n}\choose k-1}$.    Note that there is always an innermost pair of parentheses, so $N(n,k)\geq 1$ for an integer $k\in [1,n]$ and zero else. For a more in depth discussion on Narayana Numbers, see Peterson \cite{Peterson}.  We will also use the fact that $N(n,k) = N(n,n-k+1)$.  The following is well-known; since we did not find the exact reference, we give one of possible short proofs.

\begin{lemma}\label{lem:narayana}
The expected number of nestings in a random $p$-string of length $n$ is $\frac{n+1}{2}$.
\end{lemma}

\begin{proof} We perform the straightforward calculation below.

\begin{center}
$\mathbb{E}(\# {\text{ of }} \texttt{()}~) =  \sum_{k=1}^{n} k\frac{N(n,k)}{C_n}  =\frac{1}{2} \left[ \sum_{k=1}^{n} k\frac{N(n,k)}{C_n}+\sum_{k=1}^{n} k\frac{N(n,k)}{C_n}  \right] = $

$ =\frac{1}{2} \left[ 1\frac{N(n,1)}{C_n}+ 2\frac{N(n,2)}{C_n}+\quad \cdots\quad + n\frac{N(n,n)}{C_n}+  \right.$
$ \qquad + \left. n\frac{N(n,n)}{C_n} + (n-1)\frac{N(n,n-1)}{C_n} + \cdots + 1\frac{N(n,1)}{C_n} \right]. $
\end{center}

Since $N(n, k)=N(n, n-k+1)$ for $k=1, 2, ..., n$, the above is equal to

\begin{center}

$ \frac{1}{2} \left[ 1\frac{N(n,1)}{C_n}+ 2\frac{N(n,2)}{C_n}+\quad \cdots\quad + n\frac{N(n,n)}{C_n}+  \right.$
$+ \qquad  \left. n\frac{N(n,1)}{C_n} + (n-1)\frac{N(n,2)}{C_n} + \cdots + 1\frac{N(n,n)}{C_n} \right] =$
$ =\frac{1}{2} \left[ (n+1)\frac{N(n,1)}{C_n}+ (n+1)\frac{N(n,2)}{C_n}+\cdots+ (n+1)\frac{N(n,n)}{C_n}  \right]=$
$ =\frac{n+1}{2} \left[ \frac{C_n}{C_n} \right] =\frac{n+1}{2}$.
\end{center}\end{proof}


\begin{lemma}\label{ExpBigons}
Given a link $L^r_{2s-1}$, the expected number of bigons is $s+1$.
\end{lemma}

\begin{proof}
A link $L^r_{2s-1}$, is created by a string of parentheses above the axis and below the axis.  By Lemma \ref{lem:narayana}, the expected number of nestings is $\frac{s+1}{2}$ for each, and these are independent, so we add.  
\end{proof}

\begin{theorem}\label{twists}
Let $D$ be the diagram of an $(r, 2s-1)$-random meander link in $S^3$.  Then the expected number of twist regions is $(2s-1)r^2-s-1$.  In particular, when $r=1$, the expected number of twist regions is $s-2$.
\end{theorem}

\begin{proof}

Given a link diagram $D$ with $c$ crossings, $b$ bigons, and $t$ twist regions, the number $t=c-b$. Together with Lemma \ref{ExpBigons}, this observation yields 
$$t = (2s-1)r^2 - (s+1) = (2s-1)r^2-s-1.$$

Let $r = 1$ and we obtain $t=s-2$.  
\end{proof}

\subsection{Expected hyperbolic and simplicial volume.} We now apply this to obtain bounds on the mathematical expectation of volume for links.

\begin{theorem}\label{bound} [Lackenby \cite{Lackenby}, Agol-Thurston]
Let $D$ be a prime alternating diagram of a hyperbolic link $L$ in $S^3$ with twist number $t(D)$. Then
$$v_3(t(D) -2)/2 \leq Volume(S^3 -L) < 10v_3(t(D)-1),$$

where $v_3$ is the volume of a regular hyperbolic ideal 3-simplex.
\end{theorem}

\begin{remark}\label{simplicial} The above upper bound was originally proved for all hyperbolic links, including non-alternating ones. It is also shown in \cite{DT} that it holds for simplicial volume of non-hyperbolic links.
\end{remark}

In the spirit of Obeidin's observations on volume in another random link model \cite{Obeidin}, we apply this to random meander links:

\begin{corollary}\label{volume}

Let $L$ be a $(r, 2s-1)$-random meander link in $S^3$. Then the mathematical expectation for its volume, hyperbolic or simplicial, satisfies the upper bound:
$$\mathbb{E}(Volume(S^3 -L) \leq 10v_3((2s-1)r^2-s-3).$$
In particular, when $r = 1$, $L$ has $2s-1$ crossings, and
$$\mathbb{E}(Volume(S^3 -L) ) \leq 10v_3(s-3).$$
Additionally, for any $(r, 2s-1)$-random alternating meander link $L$ with $r \neq 1$, and $L$ hyperbolic, the mathematical expectation for its hyperbolic volume satisfies the lower bound: \begin{center}
$v_3((2s-1)r^2-s-5)/2 \leq \mathbb{E}(Volume(S^3 -L) ).$ \end{center}
\end{corollary}

\begin{proof} Both upper an lower bounds follow from Theorems \ref{twists} and \ref{bound}, and Remark \ref{simplicial}. Note however that the lower bound in Theorem \ref{bound} is for prime alternating diagrams only, while our model produces all alternating diagrams once the crossing information is chosen correctly. We claim that every alternating $(r, 2s-1)$-meander diagram is a prime link diagram under two conditions: with $r$ not equal to 1, and with no nugatory self-crossings. Indeed, in the absence of nugatory self-crossings, there are at least four strands going out of every non-trivial tangle when there are $r > 1$ identical copies of every strand, and hence the diagram is prime. 

A nugatory crossing can only occur in the upper rightmost or bottom leftmost end of a $(r, 2s-1)$-meander graph.  This happens when a nesting occurs at the leftmost position on the top or at the rightmost position on the bottom.  Such a link diagram is not reduced. If the diagram is alternating, one can reduce it simply by removing the nugatory self-crossings (untwisting it), and then count the remaining twists and apply the lower volume bound from Theorem \ref{bound} . To account for this potential untwisting, we subtract 2 from the expected twist number in the lower bound: this results in the constant $-5$ in our lower bound, compared to $-3$ in the bound of Theorem \ref{bound}. 

Also note that when we consider the restricted random model that produces only alternating links, this does not change Lemma \ref{ExpBigons} and Theorem \ref{twists}. Indeed, the expected number of twists is computed there for meander graphs, not involving crossing information, and hence is the same for this (restricted) model.

\end{proof}

The above bounds are all linear in $s$ and quadratic in $r$, but there are $(2s-1)r^2$ crossings in the links.  Thus, we have linear bounds in the number of crossings for the expected volume of a random meander link.

\subsection{Some related questions.}One of the difficulties with some random link and knot models is that they rarely yield hyperbolic links. For example, links obtained using random walks in plane or space are often composite, and cannot be hyperbolic by W. Thurston's results \cite{Thurston}. At the same time, hyperbolic links are often the ones that posses deep geometric and topological structure, and interesting properties.  Hyperbolicity of random meander links is therefore a natural question. 

\begin{question}
What is the probability for an $(r, 2s-1)$-random meander link diagram to represent a hyperbolic link? What about the probability for the family of alternating $(r, 2s-1)$-random meander link diagrams?
\end{question}

One way of approaching the above question might be  through tracking the presence of certain fragments in meander diagrams and meander graphs, as we did above for detecting unlinks. Note that a $(1,{2s-1})$-meander link $L$ with pierced circle $C$ is a satellite link and thus is not hyperbolic by \cite{Thurston}. Indeed, if $C$ and the axis component $A$ are linked, there is an embedded essential torus $T$ following $A$ in the complement of $L$. If $C$ and $A$ are not linked, $L$ is a split link, and is also not hyperbolic.

\begin{question}
What fragments of a meander link diagram or meander graph guarantee that the respective link complement in the 3-sphere is not hyperbolic? What is the probability for each such fragment to appear?
\end{question}

We provide bounds for mathematical expectation for volume of links complements above. In \cite{DT0, DT}, the upper bound for hyperbolic and simplicial volume from \cite{Lackenby} is refined based on differentiating between twists with 1, 2, 3, or at least 4 crossings. It is interesting whether one can find the probability of having such twists in $(r, 2s-1)$-random meander link diagram and apply this to these volume bounds similarly to the above.

\begin{question}
Can the upper bound for the mathematical expectation of the volume of $(r, 2s-1)$-random meander link complement be refined? 
\end{question}

\section{Appendix}

\begin{claim}\label{claim} With the notation as in the rest of the  paper, the following holds:
\begin{equation*} \label{eq1} \begin{split} E(s, k)   &= O(s, k) - \sum_{m=k+1}^s   \ch{m}{k} \sum_{j=m}^s (-1)^{m+j}\ch{j}{m}O(s, j)  \\
 & = O(s, k) - \sum_{m=k+1}^s O(s, m) \ch{m}{k}   \left[ \sum_{i=k+1}^{m} (-1)^{m-i} \ch{m-k}{m-i}\right]
 \end{split}
\end{equation*}
\end{claim}
\begin{proof}
Starting with the first equality 
\begin{center}

$E(s, k)   = O(s, k) - \sum_{m=k+1}^s   \ch{m}{k} \sum_{j=m}^s (-1)^{m+j}\ch{j}{m}O(s, j)$,

\end{center}
{expand the sums and group the terms with $O(s, k)$, then with $O(s, k+1)$, then $O(s, k+2)$, and so up to $O(s, s)$. We obtain}

 \noindent $ E(s, k)=O(s, k) - O(s, k+1) \ch{k+1}{k} +O(s, k+2)\left[-\ch{k+2}{k}  +   \ch{k+1}{k}\ch{k+2}{k+1}\right] +...$

\noindent  $+O(s, s) \left[-\ch{s}{k} +\ch{s-1}{k}\ch{s}{s-1} - ... \mp   \ch{k+2}{k}\ch{s}{k+2} \pm \ch{k+1}{k}\ch{s}{k+1} \right]$.

 Rewrite the latter expression using the summation notation for the sums next to $O(s, k+1)$, $O(s, k+2)$, ..., $O(s, s)$:

\noindent $E(s, k)=O(s, k)  - O(s, k+1) \sum_{i=k+1}^{k+1} (-1)^{k+1+i} \ch{i}{k}\ch{k+1}{i}-$

\noindent $-O(s, k+2) \sum_{i=k+1}^{k+2} (-1)^{k+2+i} \ch{i}{k}\ch{k+2}{i} - O(s, k+3) \sum_{i=k+1}^{k+3} (-1)^{k+3+i} \ch{i}{k}\ch{k+3}{i}-$

\noindent $-...-O(s, s) \sum_{i=k+1}^{s} (-1)^{s+i} \ch{i}{k}\ch{s}{i}=O(s, k) - \sum_{m=k+1}^s  O(s, m)  \left[ \sum_{i=k+1}^{m} (-1)^{m+i} \ch{i}{k}\ch{m}{i}\right]$.

Using that $\ch{i}{k}\ch{m}{i}=\ch{m}{k}\ch{m-k}{m-i}$, we have

\begin{equation}\label{star}
E(s, k)=O(s, k) - \sum_{m=k+1}^s O(s, m) \ch{m}{k}   \left[ \sum_{i=k+1}^{m} (-1)^{m-i} \ch{m-k}{m-i}\right].\tag{$*$}
\end{equation}
\end{proof}

\textbf{Acknowledgments}

\
Tsvietkova was partially supported by the National Science Foundation (NSF) of the United States, grants DMS-1664425 (previously 1406588) and DMS-2005496, and the Institute of Advanced Study under NSF grant DMS-1926686. Both authors thank Okinawa Institute of Science and Technology for the support.

We thank Kasper Anderson from Lund University for drawing our attention to Zeilberger's Algorithm and Poincar\'e's theorem that played an important role in Section \ref{Unlinks}. We are grateful to referees whose suggestions made this a better paper.

\

Anastasiia Tsvietkova 

Rutgers University, Newark

n.tsvet@gmail.com 

\

Nicholas Owad 

nick@owad.org

\end{document}